\documentclass[preprint,12pt]{elsarticle}
\journal{Journal of Linear Algebra and its Applications}

\usepackage{color}
\usepackage{amsmath,amsfonts,latexsym,url}
\usepackage{bm}
\usepackage[colorlinks=true,citecolor=blue,urlcolor=magenta,linkcolor=true, hypertexnames=false]{hyperref}
\usepackage{xcolor}

\newtheorem{definition}{Definition}[section]

\newtheorem{lemma}[definition]{Lemma}
\newtheorem{proposition}[definition]{Proposition}
\newtheorem{theorem}[definition]{Theorem}

\makeatletter
\def\iddots{\mathinner{\mkern1mu\raise\p@
\vbox{\kern7\p@\hbox{.}}\mkern2mu
\raise4\p@\hbox{.}\mkern2mu\raise7\p@\hbox{.}\mkern1mu}}
\makeatother

\newcommand{\PSres}{{\operatorname{PSres}}}
\newcommand{\Sres}{{\operatorname{Sres}}}
\newcommand{\Res}{{\operatorname{Res}}}
\newcommand{\Syl}{{\operatorname{Syl}}}
\newcommand{\coeff}{{\operatorname{coeff}}}
\newcommand{\chara}{{\operatorname{char}}}

\def\K{{\mathbb K}}
\def\cF{{\mathcal F}}
\def\cG{{\mathcal G}}

\def\N{{\mathbb N}}
\def\Q{{\mathbb Q}}
\def\Z{{\mathbb Z}}

\def\twoFone#1#2#3#4{{_2F_1}\biggl(\begin{matrix}
  {#1}\kern.707em {#2}\\{#3}
\end{matrix}\,\bigg|\,#4\biggr)}

\newenvironment{proof}{
  \trivlist \item[\hskip \labelsep{\it Proof.}]}{\hfill\mbox{$\Box$}
  \endtrivlist}

\begin{document}

\begin{frontmatter}

\title{Subresultants in multiple roots: an extremal case}

\author{A.~Bostan}
\address{Inria, Universit\'e Paris-Saclay, 1 rue Honor\'e d'Estienne d'Orves, 91120 Palaiseau, France.}
\ead{alin.bostan@inria.fr}
\ead[url]{http://specfun.inria.fr/bostan}

\author{C.~D'Andrea}
\address{ Departament de Matem\`atiques i Inform\`atica,
 Universitat de Barcelona (UB),
 Gran Via de les Corts Catalanes 585,
 08007 Barcelona,
 Spain} \ead{cdandrea@ub.edu}
\ead[url]{http://atlas.mat.ub.es/personals/dandrea}

\author{T.~Krick}
\address{Departamento de Matem\'atica, Facultad de
Ciencias Exactas y Naturales  and IMAS,
CONICET, Universidad de Buenos Aires,  Argentina} \ead{krick@dm.uba.ar}
\ead[url]{http://mate.dm.uba.ar/\~{}krick}

\author{A.~Szanto}
\address{Department of Mathematics, North Carolina State
University, Raleigh, NC 27695 USA}
\ead{aszanto@ncsu.edu}
\ead[url]{www4.ncsu.edu/\~{}aszanto}

\author{M.~Valdettaro}
\address{Departamento de Matem\'atica, Facultad de
Ciencias Exactas y Naturales, Universidad de Buenos Aires, Argentina}
\ead{mvaldettaro@gmail.com}

\begin{abstract} We provide explicit formulae for the coefficients of the
order-$d$ polynomial subresultant of $(x-\alpha)^m$ and $(x-\beta)^n$ with
respect to the set of Bernstein polynomials $\{(x-\alpha)^j(x-\beta)^{d-j}, \,
0\le j\le d\}$. They are given by hypergeometric expressions arising from
determinants of binomial Hankel matrices.
\end{abstract}

\date{\today}
\end{frontmatter}

\section{Introduction}

Let $\K$ be a field, and $f=f_mx^m+\cdots + f_0$ and $g=g_n x^n+\cdots +
g_0$ be two polynomials in $\K[x]$ with $f_m\neq0$ and $g_n\neq 0$. Set $0\le
d< \min\{m,n\}$. The {\em order-$d$ subresultant} $\Sres_d(f,g)$ is the
polynomial in $\K[x]$ defined as
\begin{equation}\label{srs}
\Sres_d(f,g):= \det \begin{array}{|cccccc|c}
\multicolumn{6}{c}{\scriptstyle{m+n-2d}}&\\
\cline{1-6}
f_{m} & \cdots & &\cdots &f_{d+1-(n-d-1)} &x^{n-d-1}f& \\
&  \ddots & &&\vdots  & \vdots& \scriptstyle{n-d}\\
& & f_m& \dots &f_{d+1}&f& \\
\cline{1-6}
g_{n} &\cdots & &\cdots  &g_{d+1-(m-d-1)}  &x^{m-d-1}g&\\
&\ddots && &\vdots  &\vdots  &\scriptstyle{m-d}\\
&&g_{n} & \cdots &  g_{d+1} &g&\\
\cline{1-6} \multicolumn{2}{c}{}
\end{array},
\end{equation}
where, by convention, $f_\ell = g_\ell =0$ for $\ell<0$.

Although it is not immediately transparent from the definition, $\Sres_d(f,g)$
is a polynomial of degree at most~$d$, whose coefficients are equal to some
minors of the Sylvester matrix of $f$ and $g$. Subresultants were implicitly
introduced by Jacobi~\cite{jacobi} and explicitly by
Sylvester~\cite{sylv39,sylv40}, see~\cite{GaLu03} for a comprehensive
historical account\footnote{The Sylvester matrix was defined in~\cite{sylv40},
and the order-$d$ subresultant was introduced in~\cite{sylv39,sylv40} under
the name of ``prime derivative of the $d$-degree''.}.

For any finite subsets $A=\{\alpha_1,\dots,\alpha_m\}$ and
$B=\{\beta_1,\dots,\beta_n\}$ of~$\K$, and for $0\le p\le m,\,0\le q\le n,$
one can define after Sylvester~\cite{sylv40b} the {\em double sum} expression:
\[ \Syl_{p,q}(A,B)(x):=\sum_{\substack{A^{\prime }\subset
A,\,B^{\prime}\subset
B\\|A^{\prime}|=p,\,|B^{\prime}|=q}}\frac{\mathcal{R}(A^{\prime},B^{\prime})\,
\mathcal{R}(A\backslash A^{\prime},B\backslash
B^{\prime})}{\mathcal{R}(A^{\prime},A\backslash A^{\prime
})\,\mathcal{R}(B^{\prime},B\backslash B^{\prime})}\,{\mathcal{R}(x,A^{\prime
})\,\mathcal{R}(x,B^{\prime})},\] where $ \mathcal{R}(Y,Z):=\prod_{y\in Y}
\prod_{z\in Z}(y-z). $

Sylvester stated in~\cite{sylv40b}, then proved in~\cite[Section II]{sylv},
the following connection between subresultants and double sums: assume that
$d=p+q,$ and suppose that $f$ and $g$ are the square-free polynomials
$$f=(x-\alpha_1)\cdots (x-\alpha_m) \quad \text{and} \quad g=(x-\beta_1)\cdots
(x-\beta_n).$$ Then, $${d\choose
p}\Sres_{d}(f,g)=(-1)^{p(m-d)}\Syl_{p,q}(f,g).$$

This identity can be regarded as a generalization to subresultants of the
famous Poisson formula~\cite{poisson} for the resultant of $f$ and $g$:
\begin{equation}\label{ess}
\Res(f,g)= \prod_{i=1}^m\prod_{j=1}^n(\alpha_i-\beta_j).
\end{equation}

We note however that the Poisson  formula also holds when~$f$ or $g$ have
multiple roots, since it does not involve denominators in terms of differences
of roots in subsets of $A$ or in subsets of $B$.
To demonstrate the challenges in finding closed formulae for subresultants
in the most general case, consider the instance when
$f=(x-\alpha_1)^{m_1}\cdots (x-\alpha_r)^{m_r} $, $
g=(x-\beta_1)^{n_1}\cdots (x-\beta_s)^{n_s}$ with $\alpha_i\ne
\alpha_j,\,\beta_k\ne \beta_\ell$ and $d=1$. A (quite intricate) closed formula for
$\Sres_1(f,g)$ appears in \cite[Th.2.7]{DKS13} and has the form:
{\footnotesize
\begin{align*}\Sres_1({f},{g})&=\sum_{i=1}^{
r}(-1)^{m-m_i} \Big(\prod_{\tiny \begin{array}{c}1\le j\le r\\j\neq
i\end{array}}  \frac{g(\alpha_j)^{m_j}}{
(\alpha_i-\alpha_j)^{m_j}}\Big) g(\alpha_i)^{m_i-1}
  \Big((x-\alpha_i)\cdot \\ &\sum_{\tiny\begin{array}{c}k_1+\cdots
  +\widehat k_i+\cdots \\ \cdots +k_{r+s}=m_i-1\end{array}}\prod_{\tiny \begin{array}{c}1\le j\le r\\j\neq
i\end{array}}
  \frac{{m_j-1+k_j\choose
  k_j}}{(\alpha_i-\alpha_j)^{k_j}}\prod_{1\le \ell\le s}\frac{{n_\ell-1+k_{r+\ell}\choose
  k_{r+\ell}}}{
  (\alpha_i-\beta_\ell)^{k_{r+\ell}}}\\
  &+ \min\{1,m_i-1\} \sum_{\tiny\begin{array}{c}k_1+\cdots
  +\widehat k_i+\cdots \\ \cdots +k_{r+s}=m_i-2\end{array}}\prod_{\tiny \begin{array}{c}1\le j\le r\\j\neq
i\end{array}}
  \frac{{m_j-1+k_j\choose
  k_j}}{(\alpha_i-\alpha_j)^{k_j}}\prod_{1\le \ell\le s}\frac{{n_\ell-1+k_{r+\ell}\choose
  k_{r+\ell}}}{
  (\alpha_i-\beta_\ell)^{k_{r+\ell}}}\Big).
  \end{align*}}

This is a nontrivial expression, and nothing similar has been found yet for
subresultants of general orders. It is worth noticing, however, that
determinantal formulations for subresultants of square-free polynomials
readily generalize to the case of polynomials with multiple roots (see
\cite[Th.2.5]{DKS13}), so that the difficulty seems to lie in finding expanded
expressions.

\medskip
In this article we take a completely different approach and focus
on an extremal case, which is when both $f$ and $g$ have only one multiple
root each: we get explicit expressions for $\Sres_d((x-\alpha)^m,(x-\beta)^n)$
for all $d<\min\{m,n\}$.

To do this, we set $0< d<\min\{m,n\}$ or $d=\min\{m,n\}$ when $m\ne n$ and
$c=c(m,n,d ) := m+n-2d-1$. We introduce the $d\times (d+1)$ integer Hankel
matrix with binomial entries by
{ \begin{eqnarray*}\label{Hankel}
H(m,n,d)&:=&\left(\binom{c}{m-i-j}\right)_{\substack{1\le i\le d\\0\le j\le d}}\\&=& \left(\begin{array}{ccccc}
\binom{c}{m-1}& \binom{c}{m-2}&\dots & \dots & \binom{c}{m-d-1}\\
\binom{c}{m-2}& & \iddots&\iddots & \binom{c}{m-d-2}\\
\vdots & \iddots&\iddots  & &\vdots\\
\binom{c}{m-d}&\binom{c}{m-d-1} &\dots &\dots &  \binom{c}{m-2d}
\end{array} \right),\nonumber
\end{eqnarray*}}
where, by convention, $\binom{c}{k}=0$ for $k<0$ and for $k>c$.

Denote with $q_j(m,n,d)$ the $j$-th maximal minor of $H(m,n,d)$ defined as
the determinant of the square submatrix $H_j(m,n,d)$ of $H(m,n,d)$ obtained by
deleting its $(j+1)$-th column, for $0\le j\le d.$  By convention, $q_0(m,n,0)$, the determinant of an empty matrix,
equals $1$.

Clearly, all $q_j(m,n,d)$ are integer numbers. To regard them as elements of
the field~$\K$, we consider their class via the natural ring homomorphism
$\Z\to \K$ which maps the integer 1 to the unit $1_\mathbb{K}$ of $\K$.

\medskip We now describe our main result, which provides a closed-form
expression for the coefficients of the subresultant
$\Sres_d((x-\alpha)^m,(x-\beta)^n)$ when expressed in the set of Bernstein polynomials
$\left\{ (x-\alpha)^j(x-\beta)^{d-j}, \, 0\leq j \leq d \right\}$.

\begin{theorem} \label{main} Let $m, n, d\in \N$ with $0\le d < \min\{m,n\},$ and  $\alpha,\beta\in \K$. Then,
{\small \begin{equation*}
\Sres_d((x-\alpha)^m,(x-\beta)^n)={(-1)}^{\binom{d}{2}}(\alpha-\beta)^{(m-d)(n-d)} \sum_{j=0}^d q_j(m,n,d)(x-\alpha)^j(x-\beta)^{d-j}.
\end{equation*}}
\end{theorem}

\noindent Note that Theorem \ref{main} is consistent with the Poisson formula (\ref{ess}) for $d=0$.

\smallskip Our second result completes the first one by providing explicit
expressions for the values of the minors $q_j(m,n,d),\, 0\leq j\le d$, as
products of quotients of explicit factorials.
\begin{theorem}\label{m2}
Let $m, n, d \in \N$ with $ 0<d < \min\{m,n\},$ and  $c=m+n-2d-1$. Then,
$$
q_0(m,n,d)=(-1)^{\binom{d}{2}} \displaystyle{\prod_{i=1}^{d}}\dfrac{(i-1)!\,(c+i-1)!}{(m-i-1)!(n-i)!},$$ and for
$1\le j\le d$ the following identities hold in $\mathbb{Q}$:
$$q_j(m,n,d)=
\frac{\binom{d}{j}\binom{n-d+j-1}{j}}{\binom{m-1}{j}} \, q_0(m,n,d).$$
\end{theorem}

The proof of Theorem~\ref{main} yields as a byproduct (see
Proposition~\ref{princ}) a nice description of
the $d$-th principal subresultant
$\PSres_d((x-\alpha)^m,(x-\beta)^n)$, that is, of the coefficient of $x^d$ in
$\Sres_d((x-\alpha)^m,(x-\beta)^n)$:
\begin{equation}\label{eq:PRes} \PSres_d((x-\alpha)^m,(x-\beta)^n)
= (\alpha-\beta)^{(m-d)(n-d)}\,\prod_{i=1}^{d}\frac{(i-1)!\,(c+i)!}{(m-i)!(n-i)!}.
\end{equation}

The product in~\eqref{eq:PRes} is an integer number whose prime factors are
less than $m+n-d$. Thus, if $\alpha\ne \beta$ and if the characteristic of
$\K$ is either zero or at least equal to~$m+n-d$, the subresultant $\Sres_d(
(x-\alpha)^m,(x-\beta)^n)$ is a polynomial of degree exactly~$d$. When
$\chara(\K)$ is positive but smaller than $m+n-d$, this is generally not true
(though exceptions exist, e.g., for $m=5, n=3, d=2$ and $\chara(\K)=3$). The
change of behavior might be very radical. For instance, there exist triples
$(m, n, d)$ for which the degree of $\Sres_d((x-\alpha)^m,(x-\beta)^n)$ is
less than~$d$ for any positive characteristic $p<m+n-d$. Such an example is
$(m, n, d) = (6, 8, 2)$. Another interesting example is when $p=m+n-d-1$: in
that case, $\Sres_d((x-\alpha)^m,(x-\beta)^n)$ reduces to a constant in
characteristic~$p$. In general, the degree of
$\Sres_d((x-\alpha)^m,(x-\beta)^n)$ can be determined using Theorem~\ref{m2}.
For example, in characteristic $5$, the order-8 subresultant of
$(x-\alpha)^{11}$ and $(x-\beta)^9$ is a polynomial of degree~$6$ for all
$\alpha \neq \beta$.

We briefly sketch our proof strategy for these results. We start
from the basic fact that if $\Sres_d((x-\alpha)^m,(x-\beta)^n)$ has degree exactly $d$, then any
 linear combination ${\cF}\cdot (x-\alpha)^m+{\cG}\cdot (x-\beta)^n$,
of degree bounded by~$d$ with $\deg(\cF)<n-d$ and $\deg(\cG)<m-d$, is a scalar
multiple of the subresultant (Lemma~\ref{lutil}). In Proposition \ref{hd} we
show that $\sum_{j=0}^d q_j(m,n,d)(x-\alpha)^j(x-\beta)^{d-j}$
can be expressed as such a linear combination, and determine the scalar
multiple which is the ratio between this expression and the subresultant.
Theorem~\ref{main} then follows by specializing the ``generic case''
(in characteristic zero) to fields of positive characteristic.

To prove Theorem \ref{m2}, we proceed in two main steps. We first evaluate
$q_0(m,n,d)$ in Lemma \ref{lemmq0} by using a result due to Ostrowski
(Lemma~\ref{ost}) for the determinant evaluation of a Hankel matrix involving
binomial coefficients. Then, in Lemma~\ref{kjmnd} we reduce the computation of
the remaining $q_j(m,n,d), \, 1\leq j\leq d$, to a hypergeometric identity due
to Pfaff and Saalsch\"utz (Lemma~\ref{PS}).

Sylvester's original motivation for deriving expressions in roots for
subresultants was to understand how Sturm's method for computing the number of
real roots of a polynomial in a given interval works formally; see
\cite{sylv}, where Sylvester applies the theory of subresultants developed
there to the case when $g=f',$ with $f$ having simple roots. Furthermore,
Sylvester's formulae opened the door to have great flexibility in the
evaluation of resultants and subresultants (see the book of Jouanolou and
Ap\'ery~\cite{AJ} for several ingenious formulae for the simple roots case).
The search of explicit expressions for subresultants of polynomials having
multiple roots is an active area of research; see for instance \cite{Hon99,
LP03, DHKS07, DHKS09,RS11,DKS13}. The quest for such formulae has uncovered
some interesting connections of subresultants to other well-known objects and
thus several new applications were discovered.

One of these applications are closed expressions for various rational
interpolation problems, including the Cauchy interpolation or the osculatory
rational interpolation problem~\cite{BL00,DKS15}. The search for formulae in
multiple roots also uncovered the close connection of subresultants to
multivariate symmetric Lagrange interpolation~\cite{KSV17}; generalizations to
symmetric Hermite interpolation are the topic of ongoing research. These
formulae in roots may be used to analyze the vanishing of the coefficients of
the subresultants (see the discussion after Theorem \ref{m2}), a question
related to the understanding of the performance of the Euclidean Algorithm for
polynomials over finite fields~\cite{MG1990}. Our closed formulae also led us
think about accelerating the computation of the subresultants in our
particular extremal case; this will be explained in detail in a forthcoming
paper.

\smallskip The paper is organized as follows: We first derive Theorem~\ref{m2}
in Section~\ref{2} thanks to Lemmas \ref{lemmq0} and \ref{kjmnd}. Section
\ref{3} then introduces the aforementioned multiple of the subresultant and
proves Theorem~\ref{main}.

\subsection*{Acknowledgements} This project started when the last four authors
met at the FoCM Conference in Montevideo in December 2014 and at the
University of Buenos Aires in September 2015. We are grateful to Professor
Richard A. Brualdi, who suggested a joint collaboration with the first author.
 Carlos D'Andrea was partially supported by ANPCyT
PICT-2013-0294, and the MINECO research project MTM2013-40775-P, Teresa Krick and Marcelo Valdettaro
are partially supported by ANPCyT PICT-2013-0294 and UBACyT
2014-2017-20020130100143BA, and Agnes Szanto was partially supported by NSF
grant CCF-1217557.

\bigskip \section{Proof of Theorem~\ref{m2}}\label{2} All along this section,
we work over the rational numbers to compute the coefficients
$q_j(m,n,d)$ which appear in the expression of the subresultant given
in Theorem \ref{main} over a field $\K$ of characteristic zero. As these
numbers are integers, we can regard them as elements of any field~$\K$ via
the natural ring homomorphism $\Z\to\K$ which maps $1_\mathbb{Z}\mapsto
1_\mathbb{K}.$

We start by recalling Ostrowski's determinant evaluation (Lemma \ref{ost}) for
Hankel matrices with binomial coefficients entries, and the Pfaff-Saalsch\"utz
identity (Lemma~\ref{PS}) for the evaluation at the point 1 of a special
family of $_3F_2$ hypergeometric functions.

\begin{lemma}[\cite{Ost1964}]\label{ost}
For $\ell,k\in\N$ and $a_0, a_1, \ldots, a_k\in \N$,
 \begin{equation*}
 \det\left(\binom{\ell}{a_i-j}\right)_{\substack{0\le i,j\le k}}=\ell!^{k+1}\frac{\prod_{i=1}^{k}(\ell+i)^{k+1-i}\prod_{0\le i< i'\le k}(a_{i'}-a_i)  }{\prod_{i=0}^k a_i! \prod_{i=0}^k (\ell+k-a_i)! }.
 \end{equation*}
\end{lemma}

\begin{lemma}[\cite{pfaff,Saa1890,Andrews96,Andrews97}, {\cite[\S2.3.1]{Slater66}}]\label{PS} Let
$x,y,z$ be indeterminates over $\Q$. Then, for any $k\in\mathbb{N}$,
the following identity holds in $\Q(x,y,z)$:
\begin{equation*}\sum_{j=0}^k\frac{(x)_j (y)_j
(-k)_j}{(z)_j(1+x+y-z-k)_j \, j!}= \frac{(z-x)_k(z-y)_k}{(z)_k(z-x-y)_k}.
\end{equation*}
Here $(x)_0:=1 $ and $(x)_j:=x(x+1)\cdots (x+j-1)$ for
$j\ge 1$ denotes the $j-$th Pochhammer symbol of $x$.

\end{lemma}

By applying these two results, Theorem~\ref{m2} follows straightforwardly from
Lemmas \ref{lemmq0} and \ref{kjmnd} below. Lemma \ref{lemmq0} computes
$q_0(m,n,d)$ as a direct consequence of Ostrowski's determinant
evaluation. Lemma \ref{kjmnd} computes all $q_j(m,n,d)$ for $j>0$, and is a
consequence of a binomial identity (given in~\eqref{eq:binomial-PfSa}) which is, in
fact, the Pfaff-Saalsch\"utz identity in disguise. Recall that we have set $c=m+n-2d-1$.

\begin{lemma}\label{lemmq0} Let $d, m,n\in \N$ with $ 0< d < \min\{m,n\}.$  Then,
\begin{equation*}\label{qq0}
q_0(m,n,d)=(-1)^{\binom{d}{2}}\prod_{i=1}^{d}\frac{(i-1)!(c+i-1)!}{(m-i-1)!(n-i)!}.
\end{equation*}
\end{lemma}
 \begin{proof}
Using Lemma \ref{ost} with $k=d-1$ and $a_i=m-i-2$ for $0\le i\le d-1$,
and $\ell = c $ we get
\begin{align*}q_0(m,n,d)&= \det\left(\binom{c}{m-i-j}\right)_{\substack{1\le i,j\le d}} = \det\left(\binom{c}{m-i-j-2}\right)_{\substack{0\le i,j\le d-1}} \\&=c!^d\frac{\prod_{i=1}^{d-1}(c+i)^{d-i}\prod_{0\le i<i'\le d-1 } (i-i')}{\prod_{i=0}^{d-1}  (m-i-2)!  \prod_{i=0}^{d-1} (c+d-1-(m-i-2))!}
\\&=\frac{\prod_{i=1}^{d} \left(c! \prod_{j=1}^{i-1} (c+j)\right) \cdot (-1)^{\binom{d}{2}} \prod_{i=1}^d (i-1)!}{\prod_{i=1}^{d} (m-i-1)! \prod_{i=1}^d (n-d+i-1)!}.
 \end{align*}
The statement follows by  rearranging terms.
\end{proof}

\begin{lemma} \label{kjmnd} Let $j, d, m,n\in \N$ with $0< j\le d < \min\{m,n\}.$ Then,
\begin{equation*}
 q_j(m,n,d)= \frac{\binom{d}{j}\binom{n-d+j-1}{j}}{\binom{m-1}{j}}\,q_0(m,n,d).
\end{equation*}
\end{lemma}

\begin{proof}
Observe that the matrix $H$ has full rank $d$ since by Lemma \ref{lemmq0},
its minor $q_0(m,n,d)$ is non-zero. Therefore, an elementary linear algebra
argument shows that the kernel of the induced linear map $H$: $\Q^{d+1}\to
\Q^d$ has dimension $1,$ and is generated by the (non-zero) vector $${\bf
q}(m,n,d):=(q_0(m,n,d),-q_1(m,n,d),\dots,(-1)^dq_d(m,n,d)).$$ Set
$k_j(m,n,d):= \dfrac{\binom{d}{j}\binom{n-d+j-1}{j}}{\binom{m-1}{j}}$ for $0\le j\le d$. It
suffices then to show that
\begin{equation}\label{zori}
{\bf k}(m,n,d):=(k_0(m,n,d),-k_1(m,n,d),\dots,(-1)^dk_d(m,n,d))\in \ker H,
\end{equation}
so then we would have ${\bf k}(m,n,d)=\lambda \,{\bf q}(m,n,d)$ with
$\lambda=1/q_0(m,n,d),$ as $k_0(m,n,d)=1$.

Therefore, to prove \eqref{zori} it is enough to check the following identities
\begin{equation}\label{id}\sum_{j=0}^d \binom{m+n-2d-1}{m-j-i} (-1)^j
k_j(m,n,d) =0 \quad \mbox{ for} \ 1\le i\le d.\end{equation}

We actually prove that a more general identity holds for any
$i\in\mathbb{N}$:
\begin{equation}\label{eq:binomial-PfSa}
\sum_{j=0}^d \binom{m+n-2d-1}{m-j-i} (-1)^j \dfrac{\binom{d}{j}\binom{n-d+j-1}{j}}{\binom{m-1}{j}} =
\dfrac{\binom{i-1}{d}\binom{m+n-d-1}{m-i}}{\binom{m-1}{d}},
\end{equation}
The expressions in \eqref{id} are then recovered by specializing $i$ to $1,\dots,d$.

The equalities in~\eqref{eq:binomial-PfSa} follow from the Pfaff-Saalsch\"utz
identity described in Lemma~\ref{PS}. Since both sides of~\eqref{eq:binomial-PfSa}  are polynomials in~$n$
(of degree at most $m-i$) it is enough to verify them for an infinite number of
values~$n$. We will show that they hold for $n\geq 2d$.

By observing that $(a+j-1)!=(a-1)!(a)_j$, $\binom{a+j-1}{j}=
\frac{(a)_j}{j!},\, (a-j)!=(-1)^j\frac{a!}{(-a)_j}$ and $\binom{a}{j}=(-1)^j
\frac{(-a)_j}{j!}$, we deduce that the left-hand side of
\eqref{eq:binomial-PfSa} is equal, for $n\geq 2d$, to
\[ \frac{(m+n-2d-1)!}{(m-i)!(n-2d+i-1)!}\\
  \cdot \sum_{j=0}^d
 \dfrac{(n-d)_j(-(m-i))_j(-d)_j}{(n-2d+i)_j(-(m-1))_jj!}.\]

To simplify the latter sum, we now apply Lemma~\ref{PS} for $k=d$,
and for $x, y, z$ specialized respectively to $n-d,-(m-i), n-2d+i,$ and get
\begin{align*}\sum_{j=0}^d\frac{(n-d)_j (-(m-i))_j
(-d)_j}{(n-2d+i)_j(-(m-1))_j \, j!}&=
\frac{(i-d)_d(m+n-2d)_d}{(n-2d+i)_d(m-d)_d}\\ & =
\frac{\binom{i-1}{d}(m+n-2d)\cdots (m+n-d-1)}{\binom{m-1}{d}(n-2d+i)\cdots
(n-d+i-1)},\end{align*}
from which~\eqref{eq:binomial-PfSa} follows immediately.
\end{proof}

\section{Proof of Theorem~\ref{main}} \label{3} To prove Theorem~\ref{main},
we make use of the following  well-known result,
which follows for instance  from Lemmas 7.7.4 and 7.7.6 in~\cite{Mishra}.

\begin{lemma}\label{lutil} Let $ m,n,d\in \N$ with $0\le d < \min\{m,n\},$ and
$f,\,g\in\K[x]$ have degrees $m$ and $n$ respectively. Assume
$\PSres_d(f,g)\neq0$. If $\cF,\cG\in\K[x]$ are such that $\deg(\cF)<n-d,\,\deg(\cG)<m-d$ and
$h=\cF\,f+\cG\,g$ is a non-zero polynomial in $\K[x]$ of degree at most $d$,
then there exists $\lambda\in\K\setminus \{0 \}$ satisfying
$$h=\lambda\cdot\Sres_d(f,g).$$
\end{lemma}

Following Lemma~\ref{lutil}, we first prove that $\Sres_d((x-\alpha)^m, (x-\beta)^n)$ has indeed
degree $d$ when $\alpha\ne \beta$ and $\chara(\K)=0$ or $\chara(\K) \geq
m+n-d$, in other words, that its principal subresultant $\PSres_d((x-\alpha)^m, (x-\beta)^n)$ is
non-zero.
We start by recalling a well-known result, which is used in the proof.
\begin{lemma}[Proposition 8.6(i) in \cite{AJ}]\label{adj}
Let $f,\,g\in \K[x]$.
Then, for any $\alpha\in \K,$
$$\Sres_d(f,g)(x+\alpha)=\Sres_d(f(x+\alpha),g(x+\alpha))(x).
$$
\end{lemma}

\begin{proposition} \label{princ} Let $d,m,n\in \N$ with $0< d <
\min\{m,n\},$ and $\alpha,\beta\in \K.$ Then,
$$ \PSres_d\left( (x-\alpha)^m,(x-\beta)^n\right)=
(\alpha-\beta)^{(m-d)(n-d)}\,\prod_{i=1}^{d}\frac{(i-1)!\,(c+i)!}{(m-i)!(n-i)!}.$$
In particular, if $\alpha\ne \beta$ and  $\chara(\K)=0$ or   $\chara(\K) \geq m+n-d$, then
\[\deg \left(\Sres_d((x-\alpha)^m,(x-\beta)^n) \right)=d.\]
\end{proposition}

\begin{proof}
\begin{align*}\PSres_d((x-\alpha)^m,(x-\beta)^n)&= \PSres_d(x^m,(x+\alpha-\beta)^n)\\ &=
\PSres_d(x^m,\sum_{j=0}^n{n\choose j}(\alpha-\beta)^{n-j}x^j).\end{align*}
Therefore, by the definition of the principal subresultant,
{\small{\begin{align*}&\PSres_d((x-\alpha)^m,(x-\beta)^n)=\\
&=\det \begin{array}{|ccccccc|c}
\multicolumn{7}{c}{\scriptstyle{m+n-2d}}&\\
\cline{1-7}
1& &  & \cdots &0 &\cdots &0& \\
&  \ddots & &  &\vdots & & \vdots& \scriptstyle{n-d}\\
&   &\ddots & & \vdots & &  \vdots&\\
& &  &1&0& \cdots &0& \\
\cline{1-7}
1 & &  &\cdots &\binom{n}{d}(\alpha-\beta)^{n-d}  &\cdots &\binom{n}{m-1}(\alpha-\beta)^{n-(m-1)}\\
&\ddots &&& \vdots&  &\vdots  &\scriptstyle{m-d}\\
&&1 &\dots&  \binom{n}{2d-m+1}(\alpha-\beta)^{n-(2d-m+1)}& \dots &\binom{n}{d}(\alpha-\beta)^{n-d}&\\
\cline{1-7} \multicolumn{2}{c}{}
\end{array}
\\
&= \det \begin{array}{|ccc|c}
\multicolumn{3}{c}{\scriptstyle{m-d}}\\
\cline{1-3}
\binom{n}{d}(\alpha-\beta)^{n-d}&\dots &\binom{n}{m-1}(\alpha-\beta)^{n-(m-1)} &\\
\binom{n}{d-1}(\alpha-\beta)^{n-(d-1)} & \ldots  &  \binom{n}{m-2}(\alpha-\beta)^{n-(m-2)} &{\scriptstyle{m-d}}\\
\vdots& &\vdots &\\
\binom{n}{2d-m+1}(\alpha-\beta)^{n-(2d-m+1)}&\dots &\binom{n}{d}(\alpha-\beta)^{n-d}&\\
 \cline{1-3}
 \multicolumn{2}{c}{}
\end{array}
\\ &= (\alpha-\beta)^{(m-d)(n-d)}  \det\left({n\choose d-i+j}
 \right)_{1\leq i,j\leq m-d} \\
&=(\alpha-\beta)^{(m-d)(n-d)} \prod_{i=1}^{ d}\frac{(i-1)!\,(c+i)!}{(m-i)!(n-i)!}.
\end{align*}}}
The third equality above follows from the ``weighted'' homogeneities of the
determinant. Indeed, by multiplying the $i$-th row in the second matrix above
by $(\alpha-\beta)^{i-1},\,1\leq i\leq m-d,$ the whole determinant gets
multiplied by $(\alpha-\beta)^{1+\cdots +(m-d-1)}=
(\alpha-\beta)^{\binom{m-d}{2}},$ but now for each $j=1,\ldots, m-d,$
column~$j$ has the same term $(\alpha-\beta)^{n-d+j-1}$ that can be factored
out, obtaining $(\alpha-\beta)^{(n-d+m-d-1)+\cdots+ (n-d+0)}=
(\alpha-\beta)^{(m-d)(n-d)+\binom{m-d}{2}}$ and one can then clear out the
spurious $(\alpha-\beta)^{\binom{m-d}{2}}),$ and the equality can be derived
from~Lemma~\ref{ost} with $\ell=n$, $k=m-d-1$ and $a_j= d+j+1$ for $0\le i\le
m-d-1.$
%
\end{proof}

We now show how to express a scalar multiple of the polynomial expression
$\sum_{j=0}^d q_j(m,n,d)(x-\alpha)^j(x-\beta)^{d-j}$ as a polynomial
combination $\cF\cdot (x-\alpha)^m+\cG\cdot (x-\beta)^n,$ with $\cF$ and $\cG$
satisfying the hypothesis of Lemma~\ref{lutil}. For this, we define for $0\le
d<\min\{m,n\},$
\begin{equation}\label{hdd}
h_d(\alpha,\beta,m,n):= (\alpha-\beta)^{c}\Big( \sum_{j=0}^d
q_j(m,n,d)(x-\alpha)^j(x-\beta)^{d-j}\Big).
\end{equation}
Note that $h_d(\alpha,\beta,m,n)\in\K[x]$ has degree bounded by $d.$
\begin{proposition}\label{hd}
Let $d, m,n\in \N$ with $0\le d < \min\{m,n\},$ and  $\alpha,\beta\in \K$.
There exist  $\cF,\,\cG\in \K[x]$ with $\deg(\cF)<n-d,\,\deg(\cG)<m-d$ such
that
$$h_d(\alpha,\beta,m,n) =\cF\cdot (x-\alpha)^m+ \cG\cdot(x-\beta)^n.$$
\end{proposition}

\begin{proof}
Set $f:=(x-\alpha)^m$ and $g:=(x-\beta)^n$, and write \begin{align*}(\alpha-\beta)^{c}&= (\alpha-x+x-\beta)^{c} =\sum_{k=0}^{c} (-1)^k\binom{c}{k}(x-\alpha)^k(x-\beta)^{c-k}.\end{align*}
Fix $0\le j\le d$. Then,
$$(\alpha-\beta)^{c}(x-\alpha)^j(x-\beta)^{d-j}= \sum_{k=0}^{c}
(-1)^k\binom{c}{k}(x-\alpha)^{k+j}(x-\beta)^{c-k+d-j}.$$
For $k+j\ge m$ the corresponding terms in the right-hand side are polynomial
multiples of $f$, with coefficient $\cF_j$ of degree bounded by
$(k+j)+(c-k+d-j)-m= n-d-1$. Similarly, for $c-k+d-j\ge n$, the corresponding
terms are multiples of $g$, with coefficient $\cG_j$ of degree bounded by
$(k+j)+ (c-k+d-j) -n= m-d-1$. The remaining terms satisfy $k+j<m$, i.e. $k<m-j$ and
$c-k+d-j< n$, i.e. $k>m-j-d-1$.

Therefore
\begin{align*}(\alpha-\beta)^{c}&(x-\alpha)^j(x-\beta)^{d-j}\\ &= \cF_j\,f+\cG_j\,g + \sum_{k=m-j-d}^{m-j-1} (-1)^k\binom{c}{k}(x-\alpha)^{k+j}(x-\beta)^{c-k+d-j}\\
&= \cF_j\,f+\cG_j\,g +\sum_{i=1}^{d} (-1)^{m-i-j}\binom{c}{m-i-j}(x-\alpha)^{m-i}(x-\beta)^{n-d+i-1}.\end{align*}
Multiplying each of these equations  by $q_j(m,n,d)$ for $0\le j\le  d$ and adding them up, we get
 { \begin{align*}h_d&(\alpha,\beta,m,n)= (\alpha-\beta)^{c}\Big( \sum_{j=0}^d q_j(m,n,d)(x-\alpha)^j(x-\beta)^{d-j}\Big)\\
 & = \cF\,f+\cG\,g \ +  \\ & \qquad  \sum_{j=0}^d\Big( \sum_{i=1}^{d} (-1)^{m-i-j}\binom{c}{m-i-j}q_j(m,n,d)(x-\alpha)^{m-i}(x-\beta)^{n-d+i-1}\Big),\end{align*}}

\noindent with $\cF:=\sum_{j=0}^dq_j(m,n,d)\cF_j $ and $\cG:=\sum_{j=0}^dq_j(m,n,d)\cG_j $. It turns out that
 \begin{align*}&\sum_{j=0}^d\Big(  \sum_{i=1}^{d} (-1)^{m-i-j}\binom{c}{m-i-j}q_j(m,n,d) \,(x-\alpha)^{m-i}(x-\beta)^{n-d+i-1}\Big) \\
\quad &=\sum_{i=1}^{d} (-1)^{m-i}(x-\alpha)^{m-i}(x-\beta)^{n-d+i-1}\left( \sum_{j=0}^d (-1)^{j}\binom{c}{m-i-j}q_j(m,n,d)\right)\\
\quad &= 0,
\end{align*}
 since, as observed in the proof of Lemma~\ref{kjmnd},
 $(q_0(m,n,d),-q_1(m,n,d),\dots,$ $ (-1)^dq_d(m,n,d))$ generates $\ker H$.\\
 Therefore
$h_d(\alpha,\beta,m,n)=\cF \cdot (x-\alpha)^m+\cG \cdot (x-\beta)^n$ with $\deg(\cF)<n-d$ and $\deg(\cG)<m-d$.
\end{proof}

We  now compute explicitly the $d$-th coefficient  of  $h_d(\alpha,\beta,m,n)$, which also   implies in particular that it has degree exactly
$d$ when  $\alpha\ne \beta$ and $\chara(\K)=0$ or $\chara(\K)\ge m+n-d$.

\begin{proposition}\label{mainHd} Let $d,m,n\in \N$ with $0< d < \min\{m,n\},$ and $\alpha,\beta\in \K.$ Then,
$$\coeff_{x^d}\big(h_d(\alpha,\beta,m,n)\big)=
 (-1)^{\binom{d}{2}}(\alpha-\beta)^c\prod_{i=1}^{ d}\frac{(i-1)!\,(c+i)!}{(m-i)!(n-i)!}.$$ For $d=0$ we have $h_0(\alpha,\beta,m,n)=(\alpha-\beta)^c$.
\end{proposition}

\begin{proof} It is clear that $\coeff_{x^d}(h_d(\alpha,\beta,m,n))= (\alpha-\beta)^{c} \sum_{j=0}^d q_j(m,n,d)$. The $d=0$ case follows from our convention that $q_0(m,n,0)=1$.

\smallskip
\noindent
We now show that $\sum_{j=1}^d q_j(m,n,d)=q_0(m+1,n,d)$, which proves the
statement by Lemma \ref{lemmq0}.

\noindent
Observe that
 $$\sum_{j=0}^d q_j(m,n,d)=\det \begin{array}{|ccc|c}
\multicolumn{3}{c}{\scriptstyle{d+1}}\\
\cline{1-3}
1 & \dots & (-1)^d&\\
\binom{c}{m-1}&  \dots & \binom{c}{m-d-1}&{\scriptstyle{d+1}}\\
\vdots & &\vdots&\\
\binom{c}{m-d}&\dots &  \binom{c}{m-2d}&\\
 \cline{1-3}
  \multicolumn{2}{c}{}
\end{array}.$$
For $0\le j\le d$ let ${\bf C}(j)$ denote the $(j+1)$-th column of the matrix above. We perform the following operations: ${\bf C}(j)+{\bf C}(j-1)\to {\bf C}(j)$ for $j=d,\dots,0$. By using the identity $\binom{c}{k-1}+\binom{c}{k}=\binom{c+1}{k},$ we get
$$\begin{array}{l}\det \begin{array}{|ccc|c}
\multicolumn{3}{c}{\scriptstyle{d+1}}\\
\cline{1-3}
1 & \dots & (-1)^d&\\
\binom{c}{m-1}&  \dots & \binom{c}{m-d-1}&{\scriptstyle{d+1}}\\
\vdots & &\vdots&\\
\binom{c}{m-d}&\dots &  \binom{c}{m-2d}&\\
 \cline{1-3}
  \multicolumn{2}{c}{}
\end{array} =\det\begin{array}{|c|ccc|c}
\multicolumn{1}{c}{\scriptstyle{1}}&\multicolumn{3}{c}{\scriptstyle{d}}\\
\cline{1-4}
1& 0 &\dots &0&{\scriptstyle{1}}\\
 \cline{1-4}
\binom{c}{m-1}&\binom{c+1}{m-1} & \dots & \binom{c+1}{m-d}&\\
\vdots&\vdots & &\vdots&{\scriptstyle{d}}\\
\binom{c}{m-d-1}&\binom{c+1}{m-d}&\dots &  \binom{c+1}{m-2d+1}&\\
 \cline{1-4}
  \multicolumn{2}{c}{}
\end{array}\\
=q_0(m+1,n,d).
\end{array}$$
 \end{proof}

We are now ready to prove Theorem \ref{main}.

\bigskip

\noindent {\it Proof of Theorem~\ref{main}.}
When $\alpha=\beta$, both sides of the expression in
Theorem~\ref{main} vanish.

\noindent Assume now that $\alpha\ne \beta$, and that $\chara(\K)=0$,
or $\chara(\K)\geq m+n-d$. Thanks to
Propositions~\ref{princ} and \ref{mainHd}, both $\Sres_d((x-\alpha)^m,(x-\beta)^n)$ and $h_d(\alpha,\beta,m,n)$ are non-zero polynomials of
degree exactly $d$.  Recall that we have set  $c=m+n-2d-1.$ Proposition~\ref{hd} and Lemma~\ref{lutil} with $\mu=1/\lambda$ then imply that

\begin{equation}\label{cd}\Sres_d((x-\alpha)^m,(x-\beta)^n)=\mu \cdot h_d(\alpha,\beta,m,n) \end{equation}
with \begin{align*}\mu &=\frac{\PSres_d((x-\alpha)^m,(x-\beta)^n)}{\coeff_{x^d}(h_d(\alpha,\beta,m,n))} \\
&= \frac{ (\alpha-\beta)^{(m-d)(n-d)}\,\prod_{i=1}^{d}\frac{(i-1)!\,(c+i)!}{(m-i)!(n-i)!}}{(-1)^{\binom{d}{2}}(\alpha-\beta)^c\prod_{i=1}^{ d}\frac{(i-1)!\,(c+i)!}{(m-i)!(n-i)!} }\\
&=(-1)^{\binom{d}{2}}(\alpha-\beta)^{(m-d)(n-d)-c}. \end{align*}
To prove these equalities, we use the identities shown in
Propositions~\ref{princ} and~\ref{mainHd}. The final identity for $\mu$ also holds when $d=0$. Plugging the expression of $h_d$
given in \eqref{hdd} in the identity \eqref{cd}, we deduce Theorem~\ref{main}
in this case.

 \smallskip \noindent In the general case, we use the fact that
Theorem~\ref{main} holds for $(x-u_\alpha)^m$ and $(x-u_\beta)^n$ in
$\K\supset\Q(u_\alpha,u_\beta),$ where $u_\alpha,\,u_\beta$ are indeterminates
over $\Q.$ As subresultants are defined via the determinant \eqref{srs}, and
in this case they actually belong to $\Z[u_\alpha, u_\beta][x],$ the
expression \eqref{main} holds after specializing $u_\alpha\mapsto
\alpha,\,u_\beta\mapsto\beta,$ and the standard ring homomorphism $\Z\to \K$.
This concludes the proof of Theorem~\ref{main}.

\hfill\mbox{$\Box$}

\bigskip
\bibliographystyle{alpha}
\def\cprime{$'$} \def\cprime{$'$} \def\cprime{$'$}

\end{document}